\documentclass{amsart}
 \newtheorem{thm}{Theorem}[section]
 \newtheorem{cor}[thm]{Corollary}
 \newtheorem{lem}[thm]{Lemma}
 \newtheorem{prop}[thm]{Proposition}
 \theoremstyle{definition}
 \newtheorem{defn}[thm]{Definition}
 \newtheorem{ex}[thm]{Example}
 \theoremstyle{remark}
 
 \numberwithin{equation}{section}
\usepackage{color}
\begin{document}

%
%
%
%
%
%
%
%
%

\title[Some results of $K-$frames and their multipliers]
 {Some results of $K$-frames and their multipliers}


\author[M.Shamsabadi]{Mitra Shamsabadi}
\address{Department of Mathematics and Computer Sciences, Hakim Sabzevari University, Sabzevar, Iran.}
\email{mi.shamsabadi@hsu.ac.ir}

\author[A.Arefijamaal]{Ali Akbar Arefijamaal}
\address{Department of Mathematics and Computer Sciences, Hakim Sabzevari University, Sabzevar, Iran.}
\email{arefijamaal@hsu.ac.ir}


\begin{abstract}
$K$-frames are strongly tools for the reconstruction elements from the range of a  bounded linear operator $K$ on a separable Hilbert space $\mathcal{H}$. In this paper, we study some properties of $K$-frames and  introduce the $K$-frame multipliers. We also focus to represent elements from the range of $K$  by $K$-frame multipliers.
\end{abstract}

\maketitle

\section{Introduction, notation and motivation}
For the first time, frames in Hilbert space were offered by Duffin and Schaeffer in 1952 and were brought to life by Daubechies, Grossman and  Meyer \cite{Gros}.
A frame allows each element in the underlying space to be written as a linear combination of the frame elements, but linear independence between the frame elements is not required. This fact has key role in applications as signal processing, image processing, coding theory and more. For more details and applications of ordinary frames see \cite{Ar13,Ar-BIMS,Bod05,Bol98,Cas00,Chr08}.  $K$-frames which recently introduced by G\v{a}vru\c{t}a are generalization of frames, in the meaning that the lower frame bound only holds for that admits to reconstruct from the range of a linear and bounded operator $K$ in a Hilbert space \cite{Gav07}.

A sequence $m:=\{m_{i}\}_{i\in I}$ of complex scalars is called \textit{semi-normalized} if there exist constants $a$ and $b$ such that
$0<a\leq |m_{i}|\leq b<\infty$, for all $i\in I$. For two sequences $\Phi:=\{\phi_{i}\}_{i\in I}$ and
 $\Psi:=\{\psi_{i}\}_{i\in I}$ in a Hilbert space $\mathcal{H}$
 and a sequence $m$ of complex scalars, the operator
  $\mathbb{M}_{m,\Phi,\Psi}:\mathcal {H}\rightarrow \mathcal{H}$ given by
\begin{eqnarray*}
\mathbb{M}_{m,\Phi,\Psi}f=\sum_{i\in I} m_{i}\langle f,\psi_{i}\rangle\varphi_{i}, \qquad  (f\in\mathcal{H}),
\end{eqnarray*}
is called a \textit{multiplier}. The sequence $m$ is called the
\textit{symbol}. If $\Phi$ and $\Psi$ are Bessel sequences for
 $\mathcal{H}$ and $m\in \ell^{\infty}$, then $\mathbb{M}_{m,\Phi,\Psi}$
  is well-defined and $\|\mathbb{M}_{m,\Phi,\Psi}\|\leq \sqrt{B_{\Phi}B
  _{\Psi}}\|m\|_{\infty}$ where $B_{\Phi}$ and $B_{\Psi}$ are Bessel bounds of $\Phi$ and $\Psi$, respectively \cite{Basic}. Frame multipliers have so applications in psychoacoustical modeling and denoising \cite{lab,maj}. Also, several generalizations of multipliers are proposed \cite{Ra,p,shams}.

 It is important to detect the inverse of a multiplier if it exists \cite{rep,inv}. Our aim is to introduce $K$-frame multipliers and apply them to reconstruct elements from the range of $K$.

Throughout this paper, we suppose that $\mathcal{H}$ is a separable Hilbert space, $I$ a countable index set and $I_{\mathcal{H}}$  the identity operator on $\mathcal{H}$. For two Hilbert spaces $\mathcal{H}_{1}$ and $\mathcal{H}_{2}$ we denote by $B(\mathcal{H}_{1},\mathcal{H}_{2})$ the collection of all bounded linear operators between $\mathcal{H}_{1}$ and $\mathcal{H}_{2}$, and we abbreviate $B(\mathcal{H},\mathcal{H})$ by $B(\mathcal{H})$. Also we denote the range of $K\in B(\mathcal{H})$ by $R(K)$ and $\pi_{V}$ denotes the orthogonal projection of $\mathcal{H}$ onto a closed subspace $V \subseteq \mathcal{H}$.

We end this section by  the following proposition.
\begin{prop}\label{k}\cite{Gav07}
Let $L_{1}\in B(\mathcal{H}_{1},\mathcal{H})$ and $L_{2}\in B(\mathcal{H}_{2},\mathcal{H})$ be two bounded operators. The following statements  are equivalent:
\begin{enumerate}
\item \label{Re1} $R(L_{1})\subset R(L_{2})$.
\item \label{Re2} $L_{1}L_{1}^{*}\leq \lambda^{2}L_{2}L_{2}^{*}$
for some $\lambda\geq 0$.
\item \label{Re3} there exists a bounded operator $X\in B(\mathcal{H}_{1},\mathcal{H}_{2})$ so that $L_{1}=L_{2}X$
\end{enumerate}
\end{prop}
  \section{$K$-frames}

Let $\mathcal{H}$ be a separable Hilbert space,  a sequence $F:=\lbrace f_{i}\rbrace_{i \in I} \subseteq \mathcal{H}$ is called a \textit{$K$-frame} for $\mathcal{H}$, if there exist constants $A, B > 0$ such that
\begin{eqnarray}\label{1}
A \Vert K^{*}f\Vert^{2} \leq \sum_{i\in I} \vert \langle f,f_{i}\rangle\vert^{2} \leq B \Vert f\Vert^{2}, \quad (f\in \mathcal{H}).
\end{eqnarray}

Clearly if $K=I_{\mathcal{H}}$, then $F$ is an ordinary frame. The constants $A$ and $B$ in $(\ref{1})$ are called lower and upper bounds of $F$, respectively. $\{f_{i}\}_{i\in I}$ is called \textit{$A$-tight}, if $A\|K^{*}f\|^{2}=\sum_{i\in I} \vert \langle f,f_{i}\rangle\vert^{2}$.
Also, if $\|K^{*}f\|^{2}=\sum_{i\in I} \vert \langle f,f_{i}\rangle\vert^{2}$ we call $F,$ a \textit{Parseval $K$-frame}. Obviously every $K$-frame is a Bessel sequence, hence similar to ordinary frames the \textit{synthesis operator} can be defined as $T_{F}: l^{2}\rightarrow \mathcal{H}$; $T_{F}(\{ c_{i}\}_{i\in I}) = \sum_{i\in I} c_{i}f_{i}$. It is a bounded operator and its adjoint which is called the \textit{analysis operator} given by $T_{F}^{*}(f)= \{ \langle f,f_{i}\rangle\}_{i\in I}$. Finally, the \textit{frame operator} is given by $S_{F}: \mathcal{H} \rightarrow \mathcal{H}$; $S_{F}f = T_{F}T_{F}^{*}f = \sum_{i\in I}\langle f,f_{i}\rangle f_{i}$. Many properties of ordinary frames are not hold for K-frames, for example, the frame operator of a K-frame is not invertible in general. It is worthwhile to mention that if $K$ has close range then $S_{F}$ from $R(K)$ onto $S_{F}(R(K))$ is an invertible operator \cite{Xiao}. In particular,
\begin{eqnarray}\label{bound S}
B^{-1} \| f\| \leq \|S_{F}^{-1}f\| \leq A^{-1}\|K^{\dag}\|^{2}\| f\|, \quad (f\in S_{F}(R(K))),
\end{eqnarray}
where $K^{\dag}$ is the pseudo-inverse of $K$. For further information in $K$-frames refer to \cite{arab,Gav07,Xiao}.
Suppose $\lbrace f_{i} \rbrace_{i\in I}$ is a Bessel sequence. Define $K:\mathcal{H}\rightarrow \mathcal{H}$ by $Ke_{i} = f_{i}$ for all $i\in I$ where $\{e_{i}\}_{i\in I}$ is an orthonormal basis of $\mathcal{H}$. By using Lemma 3.3.6 of \cite{Chr08} $K$  has a unique extension to a bounded operator on $\mathcal{H}$, so $\lbrace f_{i}\rbrace_{i\in I}$ is a $K$-frame for $\mathcal{H}$ by Corollary 3.7 of \cite{Xiao}.
Thus every Bessel sequence is a $K$-frame for some bounded operator $K$. Conversely, every frame sequence $\lbrace f_{i} \rbrace_{i\in I}$ can be considered as a $K$-frame. In fact, let $\lbrace f_{i} \rbrace_{i\in I}$ be a frame sequence with bounds $A$ and $B$, respectively and $K = \pi_{\mathcal{H}_{0}}$, where $H_{0} = \overline{span}_{i\in I}\lbrace f_{i}\rbrace$, then for every $f\in \mathcal{H}$
\begin{eqnarray*}
A\| K^{*}f \|^{2} &\leq& \sum_{i\in I} \left|\langle K^{*}f,f_{i}\rangle \right|^{2}\\ &=& \sum_{i\in I} \left|\langle f,f_{i}\rangle\right|^{2}\\
&\leq& B \| f\|^{2}.
\end{eqnarray*}
In the following proposition we state $K$-frames with respect to their synthesis operator.

\begin{prop}\label{onto}
A sequence $F=\{f_{i}\}_{i\in I}$ is a $K$-frame if and only if
\begin{eqnarray*}
T_{F}:\ell^{2}\rightarrow R(T_{F});\quad\{c_{i}\}_{i\in I}\mapsto \sum_{i\in I}c_{i}f_{i},
\end{eqnarray*}
is a well defined  and $R(K)\subseteq R(T_{F})
$.
\end{prop}
\begin{proof}
First, suppose that $F$ is a $K$-frame. Then   $T_{F}$ is
well defined and bounded by  \cite[Theorem 5.4.1]{Chr08}. Moreover, the lower $K$-frame condition follows that
\begin{eqnarray*}
A\langle KK^{*}f,f\rangle&=&A\|K^{*}f\|^{2}\\
&\leq& \|T_{F}^{*}f\|^{2}
= \langle T_{F}T_{F}^{*}f,f\rangle.
\end{eqnarray*}
Applying Theorem \ref{K} yields
\begin{eqnarray*}
R(K)\subseteq R(T_{F}).
\end{eqnarray*}
For the opposite direction, suppose that $T_F$ is a
well defined operator from $\ell^{2}$ to
$R(T_{F})$. Then
\cite[Lemma 3.1.1 ]{Chr08} shows that $F$ is a Bessel sequence. Assume that
$T_{F}^{\dag}:R(T_{F})\rightarrow\ell^{2}$ is the pseudo-inverse of
$T_{F}$. Since $R(K)\subseteq R(T_{F})$,  for every $f\in \mathcal{H}$ we obtain
\begin{eqnarray*}
Kf=T_{F}T_{F}^{\dag}Kf.
\end{eqnarray*}
This follows that
\begin{eqnarray*}
\|K^{*}f\|^{4}&=&\left|\left\langle K^{*}f,K^{*}f\right\rangle\right|^{2}\\
&=&\left|\left\langle K^{*}(T_{F}^{\dag})^{*}T_{F}^{*}f,K^{*}f\right\rangle\right|^{2}\\
&\leq&\|K^{*}\|^{2}\|T_{F}^{\dag}\|^{2}\|T_{F}^{*}f\|^{2}\|K^{*}f\|^{2}.
\end{eqnarray*}
Hence,
\begin{eqnarray*}
\frac{1}{\|T_{F}^{\dag}\|^{2}\|K\|^{2}}\|K^{*}f\|^{2}\leq
\sum_{i\in I}|\langle f,f_{i}\rangle|^{2}.
\end{eqnarray*}
\end{proof}

\begin{defn}
Let $\{ f_{i} \}_{i\in I}$ be a Bessel sequence. A Bessel sequence $\{ g_{i}\}_{i \in I}\subseteq \mathcal{H}$ is called a \textit{$K$-dual} of $\{ f_{i} \}_{i\in I}$ if
\begin{eqnarray}\label{dual1}
Kf = \sum_{i\in I} \langle f,g_{i}\rangle \pi_{R(K)}f_{i}, \quad (f\in \mathcal{H}).
\end{eqnarray}
\end{defn}
\begin{lem}
If  $G:=\{g_{i}\}_{i\in I}$ is a $K$-dual of a Bessel sequence $F:=\{f_{i}\}_{i\in I}$ in $\mathcal{H}$, then $\{g_{i}\}_{i\in I}$ is  a $K^{*}$-frame and $\left\{\pi_{R(K)}f_{i}\right\}_{i\in I}$ is a $K$-frame for $\mathcal{H}$.
\end{lem}
\begin{proof}
For all $f\in \mathcal{H}$ we have
\begin{eqnarray*}
\left\|Kf\right\|^{4}&=&\left|\left\langle Kf,Kf\right\rangle\right|^{2}\\
&=&\left|\left\langle \sum_{i\in I} \langle f,g_{i}\rangle \pi_{R(K)}f_{i},Kf\right\rangle\right|^{2}\\
&\leq& \sum_{i\in I}\left|\left\langle f,g_{i}\right\rangle\right|^{2} \sum_{i\in I}\left|\left\langle Kf,f_{i}\right\rangle\right|^{2}\\
&\leq& B_{F} \left\|Kf\right\|^{2}\sum_{i\in I}\left|\left\langle f,g_{i}\right\rangle\right|^{2},
\end{eqnarray*}
where $B_{F}$ is an upper bound of $\{f_{i}\}_{i\in I}$.
Hence,  $\{g_{i}\}_{i\in I}$ satisfies in the lower $K$-frame bound and
\begin{eqnarray*}
\frac{1}{B_{F}}\left\|Kf\right\|^{2}\leq \sum_{i\in I}\left|\left\langle f,g_{i}\right\rangle\right|^{2}.
\end{eqnarray*}

 In the same way, we obtain
\begin{eqnarray*}
\frac{1}{B_{G}}\left\|K^{*}f\right\|^{2}\leq \sum_{i\in I}\left|\left\langle f,\pi_{R(K)}f_{i}\right\rangle\right|^{2},
\end{eqnarray*}
where $B_{G}$ is an upper bound of $\{g_{i}\}_{i\in I}$.
\end{proof}
Now, we present a $K$-dual for every $K$-frame.
\begin{prop}
Let $K\in B(\mathcal{H})$ with closed range and $F=\{f_{i}\}_{i\in I}$ be a $K$-frame with bounds $A$ and $B$, respectively. Then $\left\{K^{*}(S_{F}\mid_{R(K)})^{-1}\pi_{S_{F}(R(K))}f_{i}\right\}_{i\in I}$ is a $K$-dual of $F$ with the bounds $B^{-1}$ and $BA^{-1}\|K\|^{2}\|K^{\dag}\|^{2}$, respectively,
\end{prop}
\begin{proof}
First note that $S_{F}|_{R(K)}:R(K)\rightarrow S_{F}(R(K))$ is invertible by (\ref{bound S}). It follows that  $\left\{K^{*}(S_{F}\mid_{R(K)})^{-1}\pi_{S_{F}(R(K))}f_{i}\right\}_{i\in I}$ is a Bessel sequence.
Moreover,
 \begin{eqnarray*}
 \left\langle  S_{F}|_{R(K)}f,g\right\rangle
&=&\left\langle \sum_{i\in I}\left\langle \pi_{R(K)}f,f_{i}\right\rangle f_{i},g\right\rangle\\
&=&\left\langle f,\sum_{i\in I}\left\langle g,f_{i}\right\rangle \pi_{R(K)}f_{i}\right\rangle,
 \end{eqnarray*}
for all $f\in R(K)$ and $g\in S_{F}(R(K))$. So,
\begin{eqnarray}\label{S^*}
(S_{F}\mid_{R(K)})^{*}g=\sum_{i\in I}\langle g,f_{i}\rangle \pi_{R(K)}f_{i}.
\end{eqnarray}

 Thus,
\begin{eqnarray*}
Kf&=&(S_{F}\mid_{R(K)})^{-1}S_{F}\mid_{R(K)}Kf\\
&=&(S_{F}|_{R(K)})^{*}\left((S_{F}\mid_{R(K)})^{-1}\right)^{*}Kf\\
&=&\sum_{i\in I}\left\langle \left((S_{F}\mid_{R(K)})^{-1}\right)^{*}Kf,f_{i}\right\rangle \pi_{R(K)}f_{i}\\
&=&\sum_{i\in I}\left\langle f,K^{*}(S_{F}\mid_{R(K)})^{-1}\pi_{S_{F}(R(K))}f_{i}\right\rangle \pi_{R(K)}f_{i},
\end{eqnarray*}
for all $f\in \mathcal{H}$. So, $\left\{K^{*}(S_{F}\mid_{R(K)})^{-1}\pi_{S_{F}(R(K))}f_{i}\right\}_{i\in I}$ is a $K$-dual of $F$ and with the lower bound of $B^{-1}$, by the last lemma. On the other hand, by using (\ref{bound S}) we have
\begin{eqnarray*}
\sum_{i\in I}\left|\left\langle f, K^{*}(S_{F}\mid_{R(K)})^{-1}\pi_{S_{F}(R(K))}f_{i}\right\rangle\right|^{2}&
\leq& B\left\|\left((S_{F}\mid_{R(K)})^{-1}\right)^{*}Kf\right\|^{2}\\
&\leq& B\left\|(S_{F}\mid_{R(K)})^{-1}\right\|^{2}\|Kf\|^{2}\\
&\leq& BA^{-1}\left\|K^{\dag}\right\|^{2}\|K\|^{2}\|f\|^{2},
\end{eqnarray*}
for all $f\in \mathcal{H}$. This completes the proof.
\end{proof}
The $K$-dual $\left\{K^{*}(S_{F}\mid_{R(K)})^{-1}\pi_{S_{F}(R(K))}f_{i}\right\}_{i\in I}$ of $F=\{f_{i}\}_{i\in I}$, introduced in the above proposition, is called the \textit{canonical $K$-dual} of $F$ and is represented by $\widetilde{F}$ for brevity.

The relation between discrete frame bounds and its canonical dual bounds is not hold for $K$-frames, see the following example.
\begin{ex}\label{ex}
Let $F=\left\{(\frac{-1}{\sqrt{2}},\frac{1}{\sqrt{2}}),
(\frac{-1}{\sqrt{2}},\frac{1}{\sqrt{2}}),(\frac{1}{\sqrt{2}},\frac{1}{\sqrt{2}})
\right\}$ in $\mathcal{H}=\mathbb{C}^{2}$ and $K$ be the orthogonal projection onto the subspace spanned by $e_{1}$, where $\{e_{1},e_{2}\}$ is the orthonormal basis of $\mathbb{C}^{2}$. For all $f=(a,b)\in \mathbb{C}^{2}$ we obtain
\begin{equation*}
1\|K^{*}f\|^{2}\leq\sum_{i\in I}\left|\left\langle f,f_{i}\right\rangle\right|^{2}=\frac{3}{2}(a^{2}+b^{2})-ab\leq 2\|f\|^{2}.
\end{equation*}
One can see that $S_{F}(R(K))=span(\frac{3}{2},\frac{-1}{2})$. Hence,
\begin{equation*}
\widetilde{F}=\left\{(\frac{-4}{5\sqrt{2}},0),
(\frac{2}{5\sqrt{2}},0),(\frac{-4}{5\sqrt{2}},
0)
\right\}
\end{equation*}
Therefore,
\begin{equation*}
\sum_{i\in I}\left|\left\langle f,\widetilde{f_{i}}\right\rangle\right|^{2}=\frac{36}{50}\|Kf\|^{2}.
\end{equation*}
\end{ex}
In the following theorem, we characterize all $K$-duals of a $K$-frame.
\begin{thm}
Assume that $K\in B(\mathcal{H})$ is a closed range operator and $F=\{f_{i}\}_{i\in I}$  a $K$-frame for $\mathcal{H}$. Then $\{g_{i}\}_{i\in I}$ is a $K$-dual of $F$ if and only if
\begin{equation*}
g_{i}=\widetilde{f_{i}}+\phi^{*}\delta_{i},
\end{equation*}
where $\{\delta_{i}\}_{i\in I}$ is the standard orthonormal basis of $\ell^{2}$ and $\phi\in B(\mathcal{H},\ell^{2})$ such that
$\pi_{R(K)}T_{F}\phi=0$.
\end{thm}
\begin{proof}
First, assume that
 \begin{equation*}
g_{i}=\widetilde{f_{i}}+\phi^{*}\delta_{i},
\end{equation*}
 for some $\phi\in B(\mathcal{H},\ell^{2})$, with $\pi_{R(K)}T_{F}\phi=0$.
 Then $\{g_{i}\}_{i\in I}$ is a Bessel sequence
because
\begin{eqnarray*}
\sum_{i\in I}\left|\left\langle f,g_{i}\right\rangle\right|^{2}&\leq&
\sum_{i\in I}\left|\left\langle f,\widetilde{f_{i}}+\phi^{*}\delta_{i}\right\rangle\right|^{2}\\
&\leq&\sum_{i\in I}\left|\left\langle f,\widetilde{f_{i}}\right\rangle\right|^{2}+\sum_{i\in I}\left|\left\langle f,\phi^{*}\delta_{i}\right\rangle\right|^{2}\\
&\leq& \left( BA^{-1}\left\|K^{\dag}\right\|^{2}\|K\|^{2}+\|\phi\|^{2}\right)\|f\|^{2},
\end{eqnarray*}
for all $f\in \mathcal{H}$, where $A$ is a lower bound for $\{f_{i}\}_{i\in I}$. Moreover,
\begin{eqnarray*}
\sum_{i\in I}\langle f,g_{i}\rangle \pi_{R(K)}f_{i}&=&\sum_{i\in I}\langle f,\widetilde{f_{i}}+\phi^{*}\delta_{i}\rangle \pi_{R(K)}f_{i}\\
&=&Kf+\pi_{R(K)}T_{F}\phi f=Kf.
\end{eqnarray*}
Therefore, $\{g_{i}\}_{i\in I}$ is a $K$-dual of $F$. For the reverse, let $\{g_{i}\}_{i\in I}$ be a $K$-dual of $F$. Define
\begin{equation*}
\phi=T_{g}^{*}-T_{F}^{*}\left((S_{F}\mid_{R(K)})^{-1}\right)^{*}K.
\end{equation*}
Then $\phi\in B(\mathcal{H},\ell^{2})$ and applying (\ref{S^*}) implies that
\begin{eqnarray*}
\pi_{R(K)}T_{F}\phi f&=&\pi_{R(K)}T_{F} T_{g}^{*}-\pi_{R(K)}T_{F} T_{F}^{*}\left((S_{F}\mid_{R(K)})^{-1}\right)^{*}Kf\\
&=& \sum_{i\in I}\langle f,g_{i}\rangle \pi_{R(K)}f_{i}-(S_{F}\mid_{R(K)})^{*}\left((S_{F}\mid_{R(K)})^{-1}\right)^{*}Kf
=0.
\end{eqnarray*}
Furthermore, 
\begin{eqnarray*}
\widetilde{f_{i}}+\phi^{*}\delta_{i}&=&K^{*}(S_{F}\mid_{R(K)})^{-1}
\pi_{S_{F}(R(K))}f_{i}+\phi^{*}\delta_{i}\\
&=&K^{*}(S_{F}\mid_{R(K)})^{-1}
\pi_{S_{F}(R(K))}f_{i}+T_{g}\delta_{i}-K^{*}(S_{F}\mid_{R(K)})^{-1}
\pi_{S_{F}(R(K))}T_{F}\delta_{i}\\
&=&g_{i},
\end{eqnarray*}
for all $i\in I$.
\end{proof}
In discrete frames, every frames and its canonical dual are dual of each other. But,  it is not true for $K$-frames in general.
In Example \ref{ex}, we obtain $S_{\widetilde{F}}\mid_{R(K^{*})}=span(\frac{36}{50},0)$. It is easy computations to see that, 
\begin{eqnarray*}
K(S_{\widetilde{F}}\mid_{R(K^{*})})^{-1}
\pi_{S_{\widetilde{F}}(R(K^{*}))}(\frac{1}{\sqrt{2}},\frac{1}{\sqrt{2}})=
(\frac{50^{2}}{36^{2}\sqrt{2}},0)\neq (\frac{1}{\sqrt{2}},\frac{1}{\sqrt{2}}).
\end{eqnarray*}

\begin{prop}
If $K\in B(\mathcal{H})$ is a closed range operator and $F=\{f_{i}\}_{i\in I}$ a $K$-frame, then $\left\{K^{*}\pi_{R(K)}f_{i}\right\}_{i\in I}$ is a $K$-dual for $\left\{(S_{F}|_{R(K)})^{-1}\pi_{S_{F}(R(K))}f_{i}\right\}_{i\in I}$. 
\end{prop}
\begin{proof}
Applying the $K$-dual $\left\{\widetilde{f_{i}}\right\}_{i\in I}$ for $\{f_{i}\}_{i\in I}$ we have
\begin{equation*}
K^{*}f=\sum_{i\in I}\langle f, \pi_{R(K)}f_{i}\rangle\widetilde{f_{i}},
\end{equation*}
for all $f\in \mathcal{H}$. On the other hand, $KK^{\dag}$ is a projection on $R(K)$. So,
\begin{eqnarray*}
Kf&=&\pi_{R(K)}(K^{\dag})^{*}K^{*}Kf\\
&=&\sum_{i\in I}\left\langle Kf,\pi_{R(K)}f_{i}\right\rangle \pi_{R(K)}(K^{\dag})^{*}\widetilde{f_{i}}\\
&=&\sum_{i\in I}\left\langle Kf,\pi_{R(K)}f_{i}\right\rangle \pi_{R(K)}(S_{F}|_{R(K)})^{-1}\pi_{S_{F}(R(K))}f_{i}.
\end{eqnarray*}
for all $f\in \mathcal{H}$.
\end{proof}
\begin{thm}\label{minimal-norm}
Let $F=\{f_{i}\}_{i\in I}$ be a $K$-frame and $\sum_{i\in I}\left\langle f,\widetilde{f_{i}}\right\rangle f_{i}$ has a representation $\sum_{i\in I}c_{i}f_{i}$ for some coefficients $\{c_{i}\}_{i\in I}$, where $f\in \mathcal{H}$. Then
\begin{eqnarray*}
\sum_{i\in I}\left|c_{i}\right|^{2}=\sum_{i\in I}|\langle
f,\widetilde{f_{i}}\rangle|^{2}+\sum_{i\in I}|c_{i}-\langle
f,\widetilde{f_{i}}\rangle|^{2}.
\end{eqnarray*}
\end{thm}
\begin{proof}
First note that $K^{*}(S_{F}|_{R(K)})^{-1}\pi_{S_{F}(R(K))}S_{F}\left((S_{F}|_{R(K)})^{-1}
\right)^{*}K$ is the frame operator of the canonical $K$-dual of $F$. Indeed,
\begin{eqnarray*}
S_{\widetilde{F}}f&=& \sum_{i\in I}\left\langle f,\widetilde{f_{i}}\right\rangle \widetilde{f_{i}}\\
&=&K^{*}(S_{F}|_{R(K)})^{-1}\pi_{S_{F}(R(K))}\sum_{i\in I}\left\langle \left((S_{F}|_{R(K)})^{-1}\right)^{*}Kf,f_{i}\right\rangle f_{i}\\
&=&
K^{*}(S_{F}|_{R(K)})^{-1}\pi_{S_{F}(R(K))}S_{F}\left((S_{F}|_{R(K)})^{-1}
\right)^{*}K,
\end{eqnarray*}
 for every $f\in \mathcal{H}$.
Moreover,
\begin{eqnarray*}
\sum_{i\in I}\left|\left\langle f,\widetilde{f_{i}}\right\rangle\right|^2&=&\left\langle S_{\widetilde{F}}f,f\right\rangle\\
&=&\left\langle K^{*}(S_{F}|_{R(K)})^{-1}\pi_{S_{F}(R(K))}S_{F}\left((S_{F}|_{R(K)})^{-1}
\right)^{*}Kf,f\right\rangle\\
\end{eqnarray*}
This follows that
\begin{eqnarray*}
\sum_{i\in I}\left|c_{i}-\left\langle
f,\widetilde{f_{i}}\right\rangle\right|^{2}&=&
\sum_{i\in I}\left|c_{i}-\left\langle \left((S_{F}|_{R(K)})^{-1}
\right)^{*}Kf,f_{i}\right\rangle\right|^{2}\\
&=&\sum_{i\in I}\left|c_{i}\right|^{2}-c_{i}\left\langle f_{i},\left((S_{F}|_{R(K)})^{-1}
\right)^{*}Kf \right\rangle-\overline{c_{i}}\left\langle\left((S_{F}|_{R(K)})^{-1}
\right)^{*}Kf, f_{i}\right\rangle\\
&&+\left\langle K^{*}(S_{F}|_{R(K)})^{-1}\pi_{S_{F}(R(K))}S_{F}\left((S_{F}|_{R(K)})^{-1}
\right)^{*}Kf,f\right\rangle\\
&=&\sum_{i\in I}\left|c_{i}\right|^{2}-2\left\langle K^{*}(S_{F}|_{R(K)})^{-1}\pi_{S_{F}(R(K))}S_{F}\left((S_{F}|_{R(K)})^{-1}
\right)^{*}Kf,f\right\rangle\\
&&+
\sum_{i\in I}\left|\left\langle f,\widetilde{f_{i}}\right\rangle\right|^2\\
&=&\sum_{i\in I}\left|c_{i}\right|^{2}-\sum_{i\in I}\left|\left\langle
f,\widetilde{f_{i}}\right\rangle\right|^{2}.
\end{eqnarray*}
\end{proof}
As a consequence of Theorem 2.5.3 of \cite{Chr08} we obtain the following result.
\begin{cor}
Let $\{f_{i}\}_{i\in I}$ be a $K$-frame with the synthesis operator $T_{F}:\ell^{2}\rightarrow\mathcal{H}$. Then for every $f\in \mathcal{H}$
we have
\begin{eqnarray*}
 T_{F}^{\dag}(S_{F}\left((S_{F}|_{R(K)})^{-1}\right)^{*}Kf)=\{\langle f,\widetilde{f_{i}}\rangle\}_{i\in I},
\end{eqnarray*}
where $T_{F}^{\dag}$ is the pseudo-inverse of $T_{F}$.
\end{cor}


\section{ $K$-frame multiplier}
In this section, we introduce the notion of multiplier for $K$-frames, when $K\in B(\mathcal{H})$.
Many properties of ordinary frame multipliers, may not
hold for $K$-frame multipliers. Similar differences can be observed between frames and $K$-frames, see \cite{Xiao}.
\begin{defn}
Let $\Phi=\{\phi_{i}\}_{i\in I}$ and $\Psi=\{\psi_{i}\}_{i\in I}$ be two Bessel sequences and let  the symbol $m=\{m_{i}\}_{i\in I}\in \ell^{\infty}$. An operator $\mathcal{R}:\mathcal{H}\rightarrow \mathcal{H}$ is called a  $\textit{K-right inverse}$ of $\mathbb{M}_{m,\Phi,\Psi}$ if
\begin{eqnarray*}
\mathbb{M}_{m,\Phi,\Psi}\mathcal{R}f=Kf, \qquad (f\in \mathcal{H}),
\end{eqnarray*}
 and $\mathcal{L}:\mathcal{H}\rightarrow \mathcal{H}$ is  called a $\textit{K-left inverse}$ of $\mathbb{M}_{m,\Phi,\Psi}$ if
\begin{eqnarray*}
\mathcal{L}\mathbb{M}_{m,\Phi,\Psi}f=Kf, \qquad (f\in \mathcal{H}).
\end{eqnarray*}
Also, a $K$-inverse is a mapping in $B(\mathcal{H})$ that is both a
$K$-left and a $K$-right inverse.

\end{defn}

By using Proposition \ref{k}, we give some sufficient and necessary conditions for the $K$-right invertibility of multipliers. Also, similar to ordinary frames,
the $K$-dual systems are investigated by $K$-right inverse (resp. $K$-left inverse) of $K$-frame multipliers.

\begin{prop}\label{K}
Let $\Phi=\{\phi_{i}\}_{i\in I}$ and $\Psi=\{\psi_{i}\}_{i\in I}$ be two Bessel sequences and   $m\in \ell^{\infty}$. The following statements  are equivalent:
\begin{enumerate}
\item \label{Re1} $R(K)\subset R(\mathbb{M}_{m,\Phi,\Psi})$.
\item \label{Re2} $KK^{*}\leq \lambda^{2}\mathbb{M}_{m,\Phi,\Psi}\mathbb{M}_{m,\Phi,\Psi}^{*}$
for some $\lambda\geq 0$.
\item \label{Re3} $\mathbb{M}_{m,\Phi,\Psi}$ has a $K$-right inverse.
\end{enumerate}
\end{prop}
Now, we can show that a $K$-dual of a $K$-frame fulfills the lower frame condition.

\begin{lem}
Let $\Phi=\{\phi_{i}\}_{i\in I}$ and $\Psi=\{\psi_{i}\}_{i\in I}$ be two Bessel sequences and  $m\in \ell^{\infty}$.
\begin{enumerate}
\item\label{Re1} If $\mathbb{M}_{m,\Phi,\Psi}=K$,  then $\Phi$ and
$\Psi$ are $K$- frame and $K^{*}$-frame, respectively. In particular,
if $\mathbb{M}_{1,\Phi,\Psi}=K$, then $\Psi$ is a $K$-dual of $\Phi$.
\item\label{Re2}  If $\mathbb{M}_{m,\Phi,\Psi}$ has a $K$-right (resp. $K$-left) inverse, then $\Phi$ (resp. $\Psi$) is $K$-frame (resp. $K^{*}$-frame).
\end{enumerate}

\end{lem}
\begin{proof}
(\ref{Re1}) Let $\mathbb{M}_{m,\Phi,\Psi}=K$. Then
\begin{eqnarray*}
\|K^{*}f\|^{4}&=&
\left|\left\langle \mathbb{M}_{m,\Phi,\Psi}K^{*}f,f\right\rangle
\right|^{2}\\
&=&\left|\sum_{i\in I}m_{i}\left\langle K^{*}f,\psi_{i}\right\rangle
\left\langle\phi_{i},f\right\rangle\right|^{2}\\
&\leq&\sup_{i\in I}|m_{i}|\|K^{*}f\|^{2}B_{\Psi}\sum_{i\in I}\left|\langle\phi_{i},f\rangle\right|^{2},
\end{eqnarray*}
for every $f\in \mathcal{H}$. Therefore $\Phi$ is $K$-frame. Similarly
\begin{eqnarray*}
\|Kf\|^{4}&=&
\left|\left\langle \mathbb{M}_{m,\Phi,\Psi}^{*}Kf,f\right\rangle\right|^{2}\\
&=&\left|\sum_{i\in I}\overline{m_{i}}\langle Kf,\phi_{i}\rangle\langle \psi_{i},f\rangle\right|^{2}\\
&\leq&\sup_{i\in I}|m_{i}|\|Kf\|^{2}B_{\Phi}\sum_{i\in I}\left|\langle\psi_{i},f\rangle\right|^{2}.
\end{eqnarray*}
Namely, $\Psi$ is a $K^{*}$-frame. In particular,
\begin{eqnarray*}
Kf&=&\sum_{i\in I}\langle f,\psi_{i}\rangle \phi_{i}\\
&=&\sum_{i\in I}\langle f,\psi_{i}\rangle \pi_{R(K)}\phi_{i}\\
\end{eqnarray*}

(\ref{Re2}) Let $\mathcal{R}$ is a $K$-right inverse of $\mathbb{M}_{m,\Phi,\Psi}$. Then
\begin{eqnarray*}
\|K^{*}f\|^{2}&=&\|\mathcal{R}^{*}\mathbb{M}_{m,\Phi,\Psi}^{*}f\|^{2}\\
&\leq&\left\|\mathcal{R}^{*}\right\|^{2}\left\|\sum_{i\in I}m_{i}\langle f,\phi_{i}\rangle \psi_{i}\right\|^{2}\\
&\leq& \sup_{i\in I}|m_{i}|\left\|\mathcal{R}\right\|^{2}B_{\Psi}\sum_{i\in I}|\langle f,\phi_{i}\rangle|^{2}.
\end{eqnarray*}
The other case is similar.
\end{proof}
In the following, we discuss about $K$-left and $K$-right invertibility of
a multiplier.
\begin{thm}
Let $\Phi=\{\phi_{i}\}_{i\in I}$ and $\Psi=\{\psi_{i}\}_{i\in I}$
be two Bessel sequences. Also, let $\mathcal{L}$ (resp. $\mathcal{R}$) be a $K$-left (resp. $K$-right) inverse of $\mathbb{M}_{1,\pi_{R(K)}\Phi,\Psi}$ (resp. $\mathbb{M}_{1,\Phi,\pi_{R(K^{*})}\Psi}$). Then $\mathcal{L}K$ (resp. $\mathcal{R}K$) are of the form of multipliers.
\end{thm}
\begin{proof}
It is obvious to check that $\mathcal{L}\Phi$ is a Bessel sequence.
Also, note that   $\Psi$ is a $K$-dual  of $\mathcal{L}\pi_{R(K)}\Phi$. Indeed,
\begin{eqnarray*}
Kf&=&\mathcal{L}\mathbb{M}_{1,\pi_{R(K)}\Phi,\Psi}f\\
&=&\sum_{i\in I}\left\langle f,\psi_{i}\right\rangle\mathcal{L}\pi_{R(K)}\phi_{i}\\
&=&\sum_{i\in I}\left\langle f,\psi_{i}\right\rangle\pi_{R(K)}\mathcal{L}\pi_{R(K)}\phi_{i}
,\ \ \qquad(f\in \mathcal{H}).
\end{eqnarray*}
Now, if $\Phi^{\dag}$ is any $K$-dual of $\Phi$, then
\begin{eqnarray*}
\mathbb{M}_{1,\mathcal{L}\pi_{R(K)}\Phi,\Phi^{\dag}}f&=&\sum_{i\in I}
\left\langle f,\phi_{i}^{\dag}\right\rangle\mathcal{L}\pi_{R(K)}\phi_{i}\\
&=&\mathcal{L}\sum_{i\in I}
\left\langle f,\phi_{i}^{\dag}\right\rangle\pi_{R(K)}\phi_{i}=\mathcal{L}Kf,
\end{eqnarray*}
for all $f\in \mathcal{H}$.
For the statement for $\mathcal{R}$, for all $f\in \mathcal{H}$ we have
\begin{eqnarray*}
K^{*}f&=&\mathcal{R}^{*}\mathbb{M}_{1,\Phi,\pi_{R(K^{*})}\Psi}^{*}f\\
&=&\mathcal{R}^{*}\mathbb{M}_{1,\pi_{R(K^{*})}\Psi,\Phi}f\\
&=&\sum_{i\in I}\left\langle f,\phi_{i}\right\rangle\mathcal{R}^{*}\pi_{R(K^{*})}\psi_{i}\\
&=&\sum_{i\in I}\left\langle f,\phi_{i}\right\rangle\pi_{R(K^{*})}\mathcal{R}^{*}\pi_{R(K^{*})}\psi_{i}.
\end{eqnarray*}  
So,  $\Phi$ is a $K^{*}$-dual of $\mathcal{R}^{*}\pi_{R(K^{*})}\Psi$. Furthermore, every $K^{*}$-dual $\Psi^{\dag}$ of $\Psi$  yields 
\begin{eqnarray*}
\mathbb{M}_{1,\Psi^{\dag},\mathcal{R}^{*}\pi_{R(K^{*}})\Psi} f
&=&\sum_{i\in I}\left\langle \mathcal{R}f,\pi_{R(K^{*})}\psi_{i}\right\rangle \psi_{i}^{\dag}\\
&=&\mathcal{R}Kf.
\end{eqnarray*}

\end{proof}
A sequence $F=\{f_{i}\}_{i\in I}$ of $\mathcal{H}$ is called
a \textit{minimal $K$-frame} whenever  it is a $K$-frame and for each $\{c_{i}\}_{i\in I}
\in \ell^{2}$ such that $\sum_{i\in I}c_{i}f_{i}=0$ then $c_{i}=0$
for all $i\in I$.
A minimal $K$-frame and its canonical $K$-dual are not
biorthogonal in general. To see this,
let $\mathcal{H}=\mathbb{C}^{4}$ and $\{e_{i}\}_{i=1}^{4}$ be the standard orthonormal basis of $\mathcal{H}$. Define $K:\mathcal{H}\rightarrow\mathcal{H}$ by
\begin{eqnarray*}
K\sum_{i=1}^{4}c_{i}e_{i}=c_{1}e_{1}+c_{1}e_{2}+c_{2}e_{3}.
\end{eqnarray*}
Then $K\in B(\mathcal{H})$ and the sequence $F=\{e_{1},e_{2},e_{3}\}$
is a minimal $K$-frame with  the bounds $A=\frac{1}{8}$ and $B=1$. It is easy to see that $\widetilde{F}=\{e_{1},e_{1},e_{2}\}$ is the canonical $K$-dual of $F$ and $\langle f_{1},\widetilde{f}_{2}\rangle\neq0$.
However, every minimal Bessel sequence, and therefore every minimal $K$-frame has a biorthogonal sequence in $\mathcal{H}$ by Lemma 5.5.3 of \cite{Chr08}.
It is worthwhile to mention that a minimal $K$-frame may have more than one biorthogonal sequence in $\mathcal{H}$, but it is
unique in $\overline{span}_{i\in I}\{f_{i}\}$.

Let $\Phi=\{\phi_{i}\}_{i\in I}$  be a $K$-frame and  $\Psi=\{\psi_{i}\}_{i\in I}$ a minimal $K^{*}$-frame.
Then $\mathbb{M}_{1,\pi_{R(K)}\Phi,\Psi}$ (resp. $\mathbb{M}_{1,\Psi,\pi_{R(K)}\Phi}$) has  a $K$-right inverse (resp. $K^{*}$-left inverse ) in the form of multipliers.
Indeed, if $G:=\{g_{i}\}_{i\in I}$ is a biorthogonal sequence for minimal $K^{*}$-frame $\Psi$, then
\begin{eqnarray*}
\mathbb{M}_{1,\pi_{R(K)}\Phi,\Psi}\mathbb{M}_{1,G,\widetilde{\Phi}}f&=&
\sum_{i,j\in I}\langle f,\widetilde{\phi_{i}}\rangle\langle g_{i},\psi_{j}\rangle \pi_{R(K)}\phi_{j}\\
&=&\sum_{i\in I}\langle f,\widetilde{\phi_{i}}\rangle \pi_{R(K)}\phi_{i}=Kf,
\end{eqnarray*}
for all $f\in \mathcal{H}$. Similarly
\begin{eqnarray*}
\mathbb{M}_{1,\widetilde{\Phi},G}\mathbb{M}_{1,\Psi,\pi_{R(K)}\Phi}f&=&
\sum_{i,j\in I}\left\langle f,\pi_{R(K)}\phi_{i}\right\rangle\left\langle \psi_{i},g_{j}\right\rangle \widetilde{\phi}_{j}\\
&=&\sum_{i\in I}\left\langle f,\pi_{R(K)}\phi_{i}\right\rangle \widetilde{\phi_{i}}
=K^{*}f.
\end{eqnarray*}
We  use the following lemma for the invertibility of operators, whose proof is left to the reader.
\begin{lem}\label{Inver}
Let $\mathcal{H}_{1}$ and $\mathcal{H}_{2}$  be two Hilbert spaces and $T\in B(\mathcal{H}_{1},\mathcal{H}_{2})$ be invertible. Suppose $U\in B(\mathcal{H}_{1},\mathcal{H}_{2})$ such that  $\|T-U\|<\|T^{-1}\|^{-1}$. Then $U$ is also invertible.
\end{lem}
 In the rest of this section we state a sufficient condition for the $K$-right invertibility of $\mathbb{M}_{m,\Psi,\Phi}$, whenever $\Psi$ is a
perturbation of $\Phi$.
\begin{thm}\label{right}
Let  $\Phi=\{\phi_{i}\}_{i\in I}$ be a $K$-frame  with
 bounds $A$ and $B$, respectively, and $\Psi=\{\psi_{i}\}_{i\in I}$ be a Bessel sequence such that
\begin{eqnarray}\label{kright}
\left(\sum_{i\in I}\left|\langle f,\psi_{i}-\phi_{i}\rangle\right|^{2}\right)^{\frac{1}{2}}
\leq \frac{aA}{b\sqrt{B}\|K^{\dag}\|^{2}}\|f\|,\qquad (f\in R(K)),
\end{eqnarray}
where $m=\{m_{i}\}_{i\in I}$ is a semi-normalized sequence with bounds $a$ and $b$, respectively.
Then
\begin{enumerate}
\item\label{Re1}  The sequence  $\Psi$  has a $K$-dual. In particular, it is a  $K$-frame.
 \item\label{Re2}  $\mathbb{M}_{\overline{m},\Psi,\Phi}$ has a $K$-right inverse in the form of multipliers.
\end{enumerate}
\end{thm}
\begin{proof}
(\ref{Re1})
Obviously $\Phi^{\dag}:=\{\sqrt{m_{i}}\phi_{i}\}_{i\in I}$ is a $K$-frame for $\mathcal{H}$ with bounds $aA$ and $bB$, respectively. Denote its frame operator by $S_{\Phi^{\dag}}$. Due to (\ref{bound S}) we obtain
 $\left\|S_{\Phi^{\dag}}^{-1}\right\|\leq \frac{\|K^{\dag}\|^{2}}{aA}$.  Moreover, (\ref{kright}) follows that
\begin{eqnarray*}
\|\mathbb{M}_{m,\Phi,\Psi}f-S_{\Phi^{\dag}}f\|&=&\left\|\sum_{i\in I}
m_{i}\left\langle f,\psi_{i}-\phi_{i}\right\rangle\phi_{i}\right\|\\
&\leq & b\left(\sum_{i\in I}\left|\langle f,\psi_{i}-\phi_{i}\rangle\right|^{2}\right)^{\frac{1}{2}}
\sqrt{B}\\
&\leq& \frac{aA}{\|K^{\dag}\|^{2}}\|f\|\\
&\leq& \frac{1}{\left\|S_{\Phi^{\dag}}^{-1}\right\|}\|f\|,
\end{eqnarray*}
for all $f\in R(K)$. Then  $\mathbb{M}_{m,\Phi,\Psi}$ has an inverse on $R(K)$,  denoted by $\mathbb{M}^{-1}$, by using Lemma \ref{Inver}.
Also, for $\mathbb{M}_{m,\Phi,\Psi}$ on $R(K)$ we have
\begin{eqnarray*}
 \left\langle (\mathbb{M}_{m,\Phi,\Psi})^{*}f,g\right\rangle&=&
 \left\langle f, \mathbb{M}_{m,\Phi,\Psi}\pi_{R(K)}g\right\rangle\\
 &=& \left\langle f, \sum_{i\in I}{m_{i}}\left\langle \pi_{R(K)}g,\psi_{i}\right\rangle
 \phi_{i}\right\rangle\\
  &=& \left\langle f, \sum_{i\in I}{m_{i}}\left\langle g,\pi_{R(K)}\psi_{i}\right\rangle
 \phi_{i}\right\rangle\\
 &=& \left\langle \sum_{i\in I}\overline{m_{i}}\left\langle f,\phi_{i}\right\rangle \pi_{R(K)}\psi_{i},g\right\rangle,
  \end{eqnarray*}
for all $f\in \mathbb{M}_{m,\Phi,\Psi}(R(K))$   and $g\in R(K)$.
     Using this fact, we obtain that
\begin{eqnarray*}
Kf&=&(\mathbb{M}^{-1}\mathbb{M}_{m,\Phi,\Psi})^{*}Kf\\
&=& \mathbb{M}_{m,\Phi,\Psi}^{*}\pi_{\mathbb{M}_{m,\Phi,\Psi}(R(K))}
(\mathbb{M}^{-1})^{*}Kf\\
&=&\sum_{i\in I}\overline{m_{i}}\left\langle \pi_{\mathbb{M}_{m,\Phi,\Psi}(R(K))} (\mathbb{M}^{-1})^{*}Kf,\phi_{i}\right\rangle\pi_{R(K)}\psi_{i}\\
&=&\sum_{i\in I}\left\langle f,K^{*}\mathbb{M}^{-1}
\pi_{\mathbb{M}_{m,\Phi,\Psi}(R(K))}m_{i}\phi_{i}\right\rangle\pi_{R(K)}\psi_{i}.
\end{eqnarray*}
Hence, $\{K^{*}\mathbb{M}_{m,\Phi,\Psi}^{-1}\pi_{\mathbb{M}_{m,\Phi,\Psi}(R(K))}m_{i}\phi
 _{i}\}_{i\in I}$ is a $K$-dual of $\Psi:=\{\psi_{i}\}_{i\in I}$.

(\ref{Re2}) The above computations  shows that  $(\mathbb{M}^{-1})^{*}K$ is a $K$-right inverse of $\mathbb{M}_{\overline{m},\Psi,\Phi}$. On the other hand, for every $K$-dual $\Phi^{d}$ of $\Phi$ we have
\begin{eqnarray*}
\mathbb{M}_{1,(\mathbb{M}^{-1})^{*}\pi_{R(K)}\Phi,\Phi^{d}}f&=&
\sum_{i\in I}\left\langle f,\phi_{i}^{d}\right\rangle (\mathbb{M}^{-1})^{*}\pi_{R(K)}\phi_{i}\\
&=&(\mathbb{M}^{-1})^{*}Kf,
\end{eqnarray*}
for all $f\in \mathcal{H}$. This completes the proof.
\end{proof}

The next theorem determines a class of multipliers which are $K$-right invertible and whose $K$-right inverse can be written as a multiplier.
\begin{thm}
Let  $\Psi=\{\psi_{i}\}_{i\in I}$ be a $K$-frame and $\Phi=\{\phi_{i}\}_{i\in I}$ a $K^{*}$-frame. Then the following assertions hold.
\begin{enumerate}
\item \label{Re1}If   $R(T^{*}_{\Psi})\subseteq R(T^{*}_{\Phi}K^{*})$, then
 $\mathbb{M}_{1,\pi_{R(K)}\Psi,\Phi}$ has a $K$-right inverse in the form of multipliers.
\item \label{Re2} If   $R(T^{*}_{\Phi})\subseteq R(T^{*}_{\Psi}K)$, then
$\mathbb{M}_{1,\Psi,\Phi}$ on $R(K^{*})$ has a $K$-left inverse in the form of multipliers.
\end{enumerate}
\end{thm}
\begin{proof}
(\ref{Re1}) First observe that the sequence $\{S_{\Phi}^{-1}\pi_{S_{\Phi}(R(K^{*}))}\phi_{i}\}_{i\in I}$, denoted by $\Phi^{\dag}$, is a Bessel sequence. Therefore,
\begin{eqnarray*}
\mathbb{M}_{1,\Phi^{\dag},\widetilde{\Psi}}f&=& \sum_{i\in I}\langle
f, K^{*}S_{\Psi}^{-1}\pi_{S_{\Psi({R(K)})}}\psi_{i}\rangle S_{\Phi}^{-1}\pi_{S_{\Phi({R(K^{*})})}}\phi_{i}\\
&=& S_{\Phi}^{-1}\pi_{S_{\Phi({R(K^{*})})}}\sum_{i\in I}\langle (S_{\Psi}^{-1})^{*}Kf,\psi_{i}\rangle \phi_{i}\\
&=& S_{\Phi}^{-1}\pi_{S_{\Phi({R(K^{*})})}}T_{\Phi}T_{\Psi}^{*}(S_{\Psi}^{-1})^{*}Kf.
\end{eqnarray*}
for all $f\in \mathcal{H}$. Now, if $R(T^{*}_{\Psi})\subseteq R(T^{*}_{\Phi}K^{*})$, then by using Proposition \ref{k}, there exists a bounded operator $X\in B(\mathcal{H})$ so that
 $T^{*}_{\Psi}=T_{\Phi}^{*}K^{*}X$. Thus,

\begin{eqnarray*}
\mathbb{M}_{1,\pi_{R(K)}\Psi,\Phi}\mathbb{M}_{1,\Phi^{\dag},
\widetilde{\Psi}}f&=&
\pi_{R(K)}T_{\Psi}T_{\Phi}^{*}S_{\Phi}^{-1}\pi_{S_{\Phi({R(K^{*})})}}T_{\Phi}
T_{\Psi}^{*}(S_{\Psi}^{-1})^{*}Kf\\
&=&\pi_{R(K)}T_{\Psi}T_{\Phi}^{*}S_{\Phi}^{-1}\pi_{S_{\Phi({R(K^{*})})}}T_{\Phi}
T_{\Phi}^{*}K^{*}X(S_{\Psi}^{-1})^{*}Kf\\
&=& \pi_{R(K)}T_{\Psi}T_{\Phi}^{*}K^{*}X(S_{\Psi}^{-1})^{*}Kf\\
&=& \pi_{R(K)}T_{\Psi}T_{\Psi}^{*}(S_{\Psi}^{-1})^{*}Kf\\
&=&S_{\Psi}^{*}(S_{\Psi}^{-1})^{*}Kf
=Kf.
\end{eqnarray*}
(\ref{Re2}) It is not difficult to see that the sequence $\{(S_{\Psi}^{-1})^{*}\pi_{R(K)}\psi_{i}\}_{i\in I}$, denoted by $\Psi^{\dag}$, is a Bessel sequence. This follows that
\begin{eqnarray*}
\mathbb{M}_{1,\widetilde{\Phi},\Psi^{\dag}}f&=&
\sum_{i\in I}\left\langle f,(S_{\Psi}^{-1})^{*}\pi_{R(K)}\psi_{i}
\right\rangle KS_{\Phi}^{-1}\pi_{S_{\Phi}(R(K^{*}))}\phi_{i}\\
&=& KS_{\Phi}^{-1}\pi_{S_{\Phi}(R(K^{*}))}\sum_{i\in I}\left\langle
S_{\Psi}^{-1}\pi_{S_{\Psi}(R(K))}f,\psi_{i}\right\rangle \phi_{i}\\
&=& KS_{\Phi}^{-1}\pi_{S_{\Phi}(R(K^{*}))}T_{\Phi}T_{\Psi}^{*}
S_{\Psi}^{-1}\pi_{S_{\Psi}(R(K))}f,
\end{eqnarray*}
for all $f\in \mathcal{H}$. Now, by using Proposition \ref{k}, there exists $X\in B(\mathcal{H})$ so that $T_{\Phi}^{*}=T_{\Psi}^{*}KX$. Therefore,
\begin{eqnarray*}
\mathbb{M}_{1,\widetilde{\Phi},\Psi^{\dag}}\mathbb{M}_{1,\Psi,\Phi}K^{*}&=&
KS_{\Phi}^{-1}\pi_{S_{\Phi}(R(K^{*}))}T_{\Phi}T_{\Psi}^{*}
S_{\Psi}^{-1}\pi_{S_{\Psi}(R(K))}T_{\Psi}T_{\Phi}^{*}K^{*}\\
&=&KS_{\Phi}^{-1}\pi_{S_{\Phi}(R(K^{*}))}T_{\Phi}T_{\Psi}^{*}
S_{\Psi}^{-1}\pi_{S_{\Psi}(R(K))}T_{\Psi}T_{\Psi}^{*}KXK^{*}\\
&=&KS_{\Phi}^{-1}\pi_{S_{\Phi}(R(K^{*}))}T_{\Phi}T_{\Psi}^{*}
S_{\Psi}^{-1}\pi_{S_{\Psi}(R(K))}S_{\Psi}KXK^{*}\\
&=&KS_{\Phi}^{-1}\pi_{S_{\Phi}(R(K^{*}))}T_{\Phi}T_{\Psi}^{*}
KXK^{*}\\
&=&KS_{\Phi}^{-1}\pi_{S_{\Phi}(R(K^{*}))}T_{\Phi}T_{\Phi}^{*}K^{*}\\
&=&KS_{\Phi}^{-1}\pi_{S_{\Phi}(R(K^{*}))}S_{\Phi}K^{*}
=KK^{*}.
\end{eqnarray*}
\end{proof}

\bibliographystyle{amsplain}

\end{document}